\documentclass[11pt]{article}
\usepackage{bm} 
\usepackage{graphicx}
\usepackage{epstopdf}
\usepackage{float}
\usepackage{amsthm}
\usepackage{natbib}
\usepackage{amsmath}
\usepackage{amsfonts}
\usepackage{tikz}
\usepackage{pgfplotstable}
\usetikzlibrary{arrows,positioning,plotmarks,external,patterns,angles,
	decorations.pathmorphing,backgrounds,fit,shapes,graphs,calc,spy}
\newtheorem{theorem}{Theorem}

\newtheorem{corollary}{Corollary}[section]

\makeatletter
\begin{document}
\title{Wasserstein Statistics in One-dimensional Location-Scale Model}
\author{Shun-ichi Amari and Takeru Matsuda\\
RIKEN Center for Brain Science
}
\date{}
\maketitle

\begin{abstract}
Wasserstein geometry and information geometry are two important structures to be introduced in a manifold of probability distributions. Wasserstein geometry is defined by using the transportation cost between two distributions, so it reflects the metric of the base manifold on which the distributions are defined. Information geometry is defined to be invariant under reversible transformations of the base space.  Both have their own merits for applications. In particular, statistical inference is based upon information geometry, where the Fisher metric plays a fundamental role, whereas Wasserstein geometry is useful in computer vision and AI applications. In this study, we analyze statistical inference based on the Wasserstein geometry in the case that the base space is one-dimensional. By using the location-scale model, we further derive the $W$-estimator that explicitly minimizes the transportation cost from the empirical distribution to a statistical model and study its asymptotic behaviors. We show that the $W$-estimator is consistent and explicitly give its asymptotic distribution by using the functional delta method. The $W$-estimator is Fisher efficient in the Gaussian case. 
\end{abstract}

\section{Introduction}

Wasserstein geometry defines a divergence between two probability distributions $p(x)$ and $q(x)$, $x \in X$ by using the cost of transportation from $p$ to $q$.  Hence, it reflects the metric of the underlying manifold $X$ on which the probability distributions are defined.  Information geometry, on the hand, studies an invariant structures wherein the geometry does not change under transformations of $X$ which may change the distance within $X$.  So information geometry is constructed independently of the metric of $X$.

Both geometries have their own histories \citep[see e.g.,][]{Villani2003,Villani2009,Amari2016}.  Information geometry has been successful in elucidating statistical inference, where the Fisher information metric plays a fundamental role.  It has successfully been applied to, not only statistics, but also machine learning, signal processing, systems theory, physics, and many other fields \citep*{Amari2016}.  Wasserstein geometry has been a useful tool in geometry, where the Ricci flow has played an important role \citep*{Villani2009, LM2018}.  Recently, it has found a widened scope of applications in computer vision, deep learning, etc. \citep[e.g.,][]{FZMAP2015, ACB2917, MMC2015, PC2019}.  There have been attempts to connect the two geometries (see \citet{AKO2018, AKOC2019} and \citet{WL2019} for examples), and  \citet{LZ2019} has proposed a unified theory connecting them.  

It is natural to consider statistical inference from the Wasserstein geometry point of view \citep{LZ2019} and compare its results with information-geometrical inference based on the likelihood.  The present article studies the statistical inference based on the Wasserstein geometry from a point of view different from that of \citet{LZ2019}.  Given a number of independent observations from a probability distribution belonging to a statistical model with a finite number of parameters, we define the $W$-estimator that minimizes the transportation cost from the empirical distribution $\hat{p}(x)$ derived from observed data to the statistical model. This is the approach taken in many studies \citep[see e.g.,][]{BJGR2019,BBR2006}. In contrast, the information geometry estimator is the one that minimizes the Kullback--Leibler divergence from the empirical distribution to the model, and it is the maximum likelihood estimator.
Note that \cite{Matsuda2021} investigated predictive density estimation under the Wasserstein loss.

We use a one-dimensional (1D) base space $X={\bm{R}}^1$, and define the transportation cost equal to the square of the Euclidean distance between two points in ${\bm{R}}^1$. We give an equation for the $W$-estimator $\hat{{\bm{\theta}}}$ for a statistical model $S=\{ p(x, {\bm{\theta}}) \}$, where $p(x,{\bm{\theta}})$ is the probability density of $x$ parametrized by a vector parameter ${\bm{\theta}}$. We then focus on the location-scale model to obtain explicit solutions 
of the $W$-estimator. We analyze its behavior, proving that it is consistent and furthermore derives its asymptotic distribution. The $W$-estimator is not Fisher efficient except for the Gaussian case, but it minimizes the $W$-divergence, which is the transportation cost between the empirical distribution and the model. We may say that it is $W$-efficient in this sense.

The present $W$-estimator is different from the estimator of \citet*{LZ2019}, which is based on the Wasserstein score function. While their fundamental theory is a new paradigm connecting information geometry and Wasserstein geometry, their estimator does not minimize the $W$-divergence from the empirical one to the model.  It is an interesting problem to compare these two frameworks of Wasserstein statistics.

The present paper is organized as follows. In section 2, we introduce the $W$-estimator for a general parametric statistical model in the 1D-case. We show that the $W$-estimator uses only a linear function of the observations. In section 3, we then focus on the location-scale model.  We give an explicit form of the $W$-estimator. In section 4, we analyze the asymptotic behavior of the $W$-estimator, proving that it is Fisher efficient in the Gaussian case. We study the geometry of the location-scale model in section 5, showing that it is Euclidean \citep*{LZ2019}, although it is a curved submanifold in the function space of $W$-geometry \citep{Takatsu2011}. Finally, we prove that the maximum likelihood estimator asymptotically minimizes the transportation cost from the true distribution to the estimated one.

\section{$W$-estimator}

First, we show the optimal transportation cost of sending $p(x)$ to $q(x)$, $x \in {\bm{R}}^1$ when the transportation cost from $x$ to $y$ is $(x-y)^2$, where $x, y \in {\bm{R}}^1$.  Let $P(x)$ and $Q(x)$ be the cumulative distribution functions of $p$ and $q$, respectively, defined by
\begin{align*}
P(x) = \int^x_{-\infty} p(y) dy, \quad Q(x) = \int^x_{-\infty} q(y) dy.
\end{align*}
Then, it is known \citep{Santambrogio2015,PC2019} that the optimal transportation plan is to send mass of $p(x)$ at $x$ to $x'$ in a way that satisfies
\begin{align*}
P(x) = Q \left(x' \right).
\end{align*}
\begin{figure}
	\centering
	\includegraphics[width=8cm]{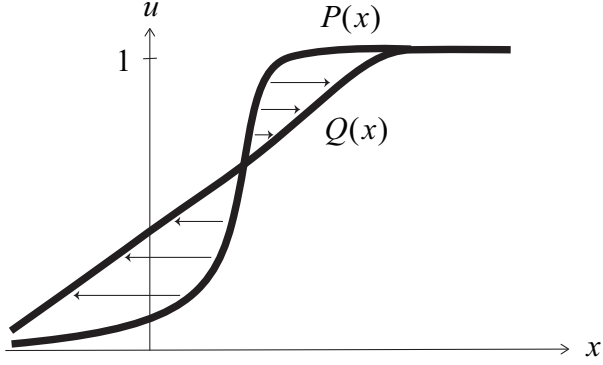}
	\caption{Optimal transportation plan from $p$ to $q$}
	\label{fig1}
\end{figure}
See Fig.~\ref{fig1}.  Thus, the total cost sending $p$ to $q$ is
\begin{align}
\label{eq:am420200228}
C(p, q) = \int^1_0 \left|
P^{-1}(u) - Q^{-1}(u)
\right|^2 du,
\end{align}
where $P^{-1}$ and $Q^{-1}$ are the inverse functions of $P$ and $Q$, respectively.

We consider a regular statistical model
\begin{align*}
S = \left\{ p(x, {\bm{\theta}}) \right\},
\end{align*}
parametrized by a vector parameter ${\bm{\theta}}$, where $p(x, {\bm{\theta}})$ is a probability density function of a random variable $x \in {\bm{R}}^1$ with respect to the Lebesgue measure of ${\bm{R}}^1$.  Let
\begin{align*}
D = \left\{ x_1, \cdots, x_n \right\}
\end{align*}
be $n$ independent samples from $p(x, {\bm{\theta}})$. We denote the empirical distribution by
\begin{align*}
\hat{p}(x) = \frac 1n
\sum_i \delta \left(x-x_i \right),
\end{align*}
where $\delta$ is the Dirac delta function. We rearrange $x_1, \cdots, x_n$ in the increasing order,
\begin{align*}
x_{(1)} \le x_{(2)} \le \cdots 
\le x_{(n)},
\end{align*}    
which are order statistics.  

The optimal transportation plan from $\hat{p}(x)$ to $p(x, {\bm{\theta}})$ is explicitly solved when $x$ is one-dimensional, $x \in {\bm{R}}^1$.  
The optimal plan is to transport mass at $x$ to those points $x’$ satisfying
\begin{align*}
\hat{P}(x_-) \leq P (x', {\bm{\theta}} ) \leq \hat{P}(x),
\end{align*}
where $\hat{P}(x)$ and $P(x, {\bm{\theta}})$ are the (right-continuous) cumulative distribution functions of $\hat{p}(x)$ and $p(x, {\bm{\theta}})$, respectively:
\begin{align*}
\hat{P}({{x}}) = \int^x_{-\infty}
\hat{p}(y) dy, \quad P(x, {\bm{\theta}}) = \int^x_{-\infty}
p(y, {\bm{\theta}}) dy,
\end{align*}
and $\hat{P}(x_-)=\lim_{y \to x-0} \hat{P}(y)$.
The total cost $C$ of optimally transporting $\hat{p}(x)$ to $p(x, {\bm{\theta}})$ is given by
\begin{align*}
C({\bm{\theta}}) = \int^1_0
\left|  
\hat{P}^{-1}(u)-P^{-1}(u, {\bm{\theta}})
\right|^2 du,
\end{align*}
where $\hat{P}^{-1}$ and $P^{-1}$ are inverse functions of $\hat{P}$ and $P$, respectively.
Note that
\begin{align*}
\hat{P}^{-1}(u) = \inf \{ y \mid P(y) \geq u \}.
\end{align*}

Let $z_0(\bm{\theta}), z_1(\bm{\theta}), \cdots, z_n(\bm{\theta})$ be the points of the equi-probability partition of the distribution $p(x, {\bm{\theta}})$ such that
\begin{align}
\label{eq:am1320200228}
\int^{z_i(\bm{\theta})}_{z_{i-1}(\bm{\theta})}
p(x, {\bm{\theta}}) dx =
\frac 1n,
\end{align}
where $z_0(\bm{\theta})=-\infty$ and $z_n(\bm{\theta})=\infty$.  In terms of the cumulative distribution, $z_i(\bm{\theta})$ can be written as
\begin{align*}
P \left(z_i(\bm{\theta}), {\bm{\theta}} \right)
= \frac in
\end{align*}
and
\begin{align*}
z_i(\bm{\theta}) = P^{-1} 
\left( \frac in, {\bm{\theta}} \right).
\end{align*}
See Fig.~\ref{fig2}.
\begin{figure}
	\centering
	\includegraphics[width=8cm]{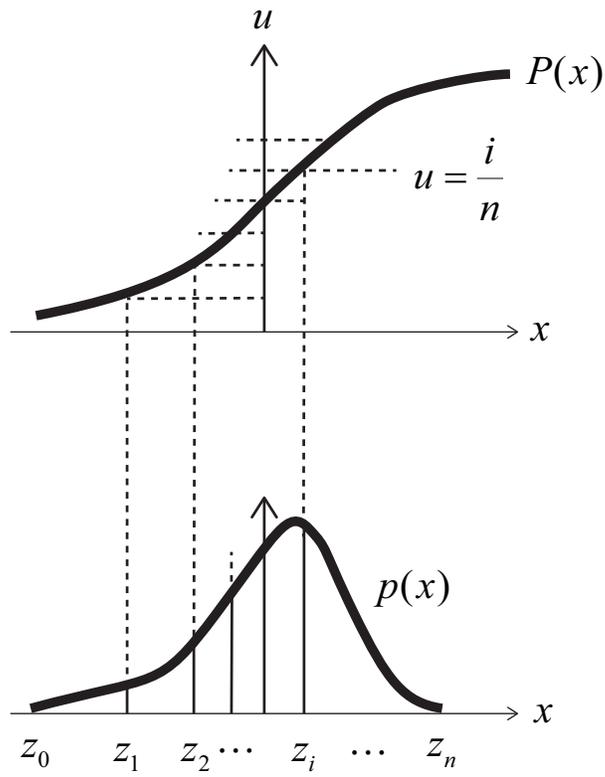}
	\caption{Equi-partition points $z_0,z_1,\dots,z_n$ of $p(x)$}
	\label{fig2}
\end{figure}

The optimal transportation cost is rewritten as
\begin{align*}
C({\bm{\theta}}) &= \sum_i
\int^{z_i(\bm{\theta})}_{z_{i-1}(\bm{\theta})}(x_{(i)}-y)^2
p(y, {\bm{\theta}})dy = \frac 1n \sum_i x^2_{(i)}-2 \sum_i k_i
({\bm{\theta}}) x_{(i)} + S({\bm{\theta}}),
\end{align*}
where we have used (\ref{eq:am1320200228}) and put
\begin{align}
\label{eq:am1820200225}
k_i({\bm{\theta}}) &=
\int^{z_i(\bm{\theta})}_{z_{i-1}(\bm{\theta})} y p 
(y, {\bm{\theta}})dy, \\
S({\bm{\theta}}) &=
\sum_i \int^{z_i(\bm{\theta})}_{z_{i-1}(\bm{\theta})} y^2 p
(y, {\bm{\theta}})dy =
\int_{-\infty}^{\infty} y^2 p(y, {\bm{\theta}})dy. \nonumber
\end{align} 
By using the mean and variance of $p(x, {\bm{\theta}})$, 
\begin{align*}
\mu({\bm{\theta}}) &= \int_{-\infty}^{\infty} y p(y, {\bm{\theta}}) dy, \\
\sigma^2({\bm{\theta}}) &= \int_{-\infty}^{\infty} y^2 p(y, {\bm{\theta}}) dy-\mu({\bm{\theta}})^2,
\end{align*}
we have
\begin{align*}
S({\bm{\theta}}) = \mu({\bm{\theta}})^2 + \sigma^2({\bm{\theta}}).
\end{align*}

The $W$-estimator $\hat{\bm{\theta}}$ is the minimizer of $C({\bm{\theta}})$.  Differentiating $C({\bm{\theta}})$ with respect to ${\bm{\theta}}$ and putting it equal to 0, we obtain the estimating equation as follows.

\begin{theorem}\upshape
	The $W$-estimator $\hat{\bm{\theta}}$ satisfies
	\begin{align}
	\label{eq:am2320200225}
	\frac{\partial}{\partial {\bm{\theta}}} \sum_i k_i ({\bm{\theta}}) x_{(i)} =
	\frac 12 \frac{\partial}{\partial {\bm{\theta}}} S({\bm{\theta}}).
	\end{align}
\end{theorem}

It is interesting to see that the estimating equation is linear in $n$ observations $x_{(1)}, \cdots, x_{(n)}$ for any statistical model.  This is quite different from the maximum likelihood estimator or Bayes estimator.

Here, we will give a rough sketch showing that the $W$-estimator is consistent; that is, it converges to the true ${\bm{\theta}}_0$ as $n$ tends to infinity \citep[see][]{BBR2006}. More detailed discussions are given for the location-scale model in the next section.  As $n$ tends to infinity, the order statistic $x_{(i)}$ converges to the $i$th partition point $z_i ({\bm{\theta}}_0)$, when the true parameter is ${\bm{\theta}}_0$.  From (\ref{eq:am1820200225}), we see that
\begin{align*}
k_i({\bm{\theta}}) \approx \frac 1n z_i({\bm{\theta}})
\end{align*}
as $n \rightarrow \infty$, so we have 
\begin{align*}
\sum_i k_i({\bm{\theta}}) x_{(i)} \approx \frac{1}{n} \sum_i z_i(\bm{\theta}) z_i(\bm{\theta}_0).
\end{align*}
Moreover, as $n$ tends to infinity,
\begin{align*}
S({\bm{\theta}})= \int_{-\infty}^{\infty} z^2 p(z, {\bm{\theta}}) dz \approx \frac 1n \sum_i z^2_i (\bm{\theta}).
\end{align*}
Therefore, ${\bm{\theta}}={\bm{\theta}}_0$ is the solution of (\ref{eq:am2320200225}), showing the consistency of the estimator.

{\textbf{Remark}}\quad \cite{BBR2006} investigated existence, measurability and consistency of the $W$-estimator for general models and \cite{BJGR2019} extended this result to mis-specified models.
\cite{MMC2015} studied $W$-estimators for Boltzmann machines.
In this study, we focus on the one-dimensional models, for which Theorem~1 gives a closed-form solution of the $W$-estimator.

\section{$W$-estimator in location-scale model}
Now, we focus on location-scale models.
Let $f(z)$ be a standard probability density function, satisfying
\begin{align*}
\int_{-\infty}^{\infty} f(z) dz &= 1, \\
\int_{-\infty}^{\infty} zf(z) dz &= 0, \\
\int_{-\infty}^{\infty} z^2 f(z) dz &= 1,
\end{align*}
that is, its mean is 0 and the variance is 1. The location-scale model $p(x, {\bm{\theta}})$ is written as
\begin{align}
\label{eq:am3120200225}
p(x, {\bm{\theta}}) =
\frac 1{\sigma} f
\left( \frac{x-\mu}{\sigma} \right),
\end{align}
where ${\bm{\theta}}=(\mu, \sigma)$ is a parameter for specifying the distribution.

We define the equi-probability partition points $z_i$ for the standard $f(z)$ as
\begin{align*}
z_i = F^{-1} \left( \frac in \right),
\end{align*}
where $F$ is the cumulative distribution function
\begin{align*}
F(z) = \int^z_{-\infty} f(x)dx.
\end{align*}

We use the following transformation of the location and scale,
\begin{align*}
z &= \frac{x-\mu}{\sigma}, \\
x &= \sigma z+ \mu.
\end{align*}
The equi-probability partition points $y_i=y_i(\bm{\theta})$ of $p(x, {\bm{\theta}})$ are given by
\begin{align*}
y_i = \sigma z_i + \mu.
\end{align*}
The cost of the optimal transport from the empirical distribution $\hat{p}(x)$ to $p(x, \bm{\theta})$ is then written as
\begin{align}
C(\mu, \sigma) &=  \sum_i
\int^{y_i}_{y_{i-1}}
\left(x_{(i)}-x \right)^2 p
(x, \mu, \sigma) dx \nonumber\\
\label{eq:am3720200225}
&=  
\mu^2+\sigma^2 + \frac{1}{n} \sum_i x^2_{(i)} -2 \sum_i
x_{(i)} \int^{z_i}_{z_{i-1}} \left(\sigma z+ \mu \right)
f(z)dz.
\end{align}
By differentiating (\ref{eq:am3720200225}), we obtain
\begin{align*}
\frac 12 \frac{\partial}{\partial \mu} C
&= \mu-\frac 1n \sum_i x_{(i)}, \\
\frac 12 \frac{\partial}{\partial \sigma}C
&= \sigma-\sum_i k_i x_{(i)},
\end{align*}
where
\begin{align}
\label{eq:am4020200225}
k_i = \int^{z_i}_{z_{i-1}} z f(z)dz,
\end{align}
which does not depend on $\mu$ or $\sigma$ and depends only on the shape of $f$.  By putting the derivatives equal to 0, we obtain the following theorem.

\begin{theorem}\label{th_West}\upshape
	The $W$-estimator of a location-scale model is given by
	\begin{align}
	\label{eq:am4120200225}
	\hat{\mu} &= \frac 1n \sum_i x_{(i)}, \\
	\label{eq:am4220200225}
	\hat{\sigma} &= \sum_i k_i x_{(i)}.
	\end{align}
\end{theorem}

{\textbf{Remark}}\quad The $W$-estimator of the location parameter $\mu$ is the arithmetic mean of the observed data irrespective of the form of $f$.  The $W$-estimator of the scale parameter $\sigma$ is also a linear function of the observed data $x_{(1)}, \cdots, x_{(n)}$, but it depends on $f$ through $k_i$.


\section{Asymptotic distribution of $W$-estimator}
Here, we derive the asymptotic distribution of the $W$-estimator in location-scale models.
Our derivation is based on the fact that the $W$-estimator has the form of L-statistics \citep{vV}, which is a linear combination of order statistics.

\begin{theorem}
	The asymptotic distribution of the $W$-estimator ($\hat{\mu},\hat{\sigma})$ in \eqref{eq:am4120200225} \eqref{eq:am4220200225} is 
	\begin{align}\label{asymp}
	\sqrt{n} \begin{pmatrix} \hat{\mu}-\mu \\ \hat{\sigma}-\sigma \end{pmatrix} \Rightarrow N \left( \begin{pmatrix} 0 \\ 0 \end{pmatrix}, \begin{pmatrix} \sigma^2 & \frac{1}{2} m_3 \sigma^2 \\ \frac{1}{2} m_3 \sigma^2 & \frac{1}{4} (m_4 - 1) \sigma^2 \end{pmatrix} \right),
	\end{align}
	where
	\begin{align*}
	m_4 = \int_{-\infty}^{\infty} z^4 f(z) dz, \quad m_3 = \int_{-\infty}^{\infty} z^3 f(z) dz,
	\end{align*}
	are the fourth and third moments of $f(z)$, respectively.
\end{theorem}
\begin{proof}
	Without loss of generality, we focus on the case $\mu=0$ and $\sigma=1$.
	Let
	\[
	\phi(\widetilde{F}) = \left( \int_0^1 \widetilde{F}^{-1}(u) {\rm d} u, \int_0^1 F^{-1}(u) \widetilde{F}^{-1}(u) {\rm d} u \right),
	\]
	where $F$ is the distribution function of $f$.
	Note that $\phi(F)=(0,1)$.
	Then, the $W$-estimator in \eqref{eq:am4120200225} \eqref{eq:am4220200225} is expressed as
	\[
	(\hat{\mu}, \hat{\sigma})=\phi(F_n),
	\]
	where $F_n$ is the empirical distribution of $x_1,\dots,x_n$, because
	\begin{align*}
	k_i = \int_{(i-1)/n}^{i/n} F^{-1}(u) {\rm d} u.
	\end{align*}
	
	To derive the asymptotic distribution of $\phi(F_n)$, we use the functional delta method \citep{vV}.
	From Donsker's theorem (Theorem 19.3 of \cite{vV}),
	\[
	\sqrt{n} (F_n-F) \Rightarrow \mathbb{G}_F = \mathbb{G} \circ F,
	\]
	where $\mathbb{G}$ is the standard Brownian bridge.
	Namely, $\mathbb{G}_F$ is the mean zero Gaussian process on $(-\infty,\infty)$ with covariance given by
	\[
	{\rm E} [\mathbb{G}_F (x) \mathbb{G}_F (y)] = F(x) \wedge F(y) - F(x)F(y),
	\]
	where $s \wedge t=\min (s,t)$.
	Let $u=F(x)$ and $x_t=(F+tH)^{-1}(u)$ for sufficiently small $t$. 
	Then, from $x_0=x$,
	\[
	u = F(x_t) + t H(x_t) = F(x) + f(x) (x_t-x) + t H(x) + O(t^2),
	\]
	which yields 
	\[
	x_t=x-t \frac{H(x)}{f(x)} +O(t^2).	
	\]
	Thus, by putting $u=F(x)$,
	\begin{align*}
	\int_0^1 (F+tH)^{-1}(u) {\rm d} u &= \int_0^1 \left( F^{-1}(u) -t \frac{H(F^{-1}(u))}{f(F^{-1}(u))} \right) {\rm d} u + O(t^2) \\
	&= \int_{-\infty}^{\infty} x f(x) {\rm d} x - t \int_{-\infty}^{\infty} H(x) {\rm d} x + O(t^2).
	\end{align*}
	Similarly,
	\begin{align*}
	\int_0^1 F^{-1}(u) (F+tH)^{-1}(u) {\rm d} u = \int_{-\infty}^{\infty} x^2 f(x) {\rm d} x - t \int_{-\infty}^{\infty} x H(x) {\rm d} x + O(t^2).
	\end{align*}
	Therefore, $\phi$ is Hadamard differentiable with derivative given by
	\[
	\phi_F'(H) = \lim_{t \to 0} \frac{\phi(F+tH)-\phi(F)}{t} = \left( -\int_{-\infty}^{\infty} H(x) {\rm d} x, -\int_{-\infty}^{\infty} x H(x) {\rm d} x \right).
	\]
	Thus, from Theorem 20.8 of \cite{vV},
	\[
	\sqrt{n} (\phi(F_n)-\phi(F)) \Rightarrow \phi_F'(\mathbb{G}_F) \sim {\rm N} (0,\Sigma),
	\]
	where
	\[
	\Sigma_{11} = \int_{-\infty}^{\infty} \int_{-\infty}^{\infty} (F(x) \wedge F(y) - F(x) F(y)) {\rm d} x {\rm d} y,
	\]
	\[
	\Sigma_{12} = \Sigma_{21} = \int_{-\infty}^{\infty} \int_{-\infty}^{\infty} x (F(x) \wedge F(y) - F(x) F(y)) {\rm d} x {\rm d} y,
	\]
	\[
	\Sigma_{22} = \int_{-\infty}^{\infty} \int_{-\infty}^{\infty} xy (F(x) \wedge F(y) - F(x) F(y)) {\rm d} x {\rm d} y.
	\]
	By using
	\begin{align*}
	\int_{-\infty}^y F(x) {\rm d} x &= \left[ (x-y)F(x) \right]_{x=-\infty}^{x=y} - \int_{-\infty}^y (x-y) f(x) {\rm d} x \\
	&= - \int_{-\infty}^y (x-y) f(x) {\rm d} x,
	\end{align*}
	\begin{align*}
	\int_x^{\infty} (y-x) (1-F(y)) {\rm d} y &= \left[ \frac{(y-x)^2}{2} (1-F(y)) \right]_{y=x}^{y=\infty} - \int_x^{\infty} \frac{(y-x)^2}{2} (-f(y)) {\rm d} y \\
	&= \int_{x}^{\infty} \frac{(y-x)^2}{2} f(y) {\rm d} y,
	\end{align*}
	and the symmetry of the integrand of $\Sigma_{11}$, we have
	\begin{align*}
	\Sigma_{11} &= 2 \int_{-\infty}^{\infty} \int_{-\infty}^{y} F(x) (1-F(y)) {\rm d} x {\rm d} y \\
	&= 2 \int_{-\infty}^{\infty} (1-F(y)) \int_{-\infty}^{y} F(x) {\rm d} x {\rm d} y \\
	&= -2 \int_{-\infty}^{\infty} (1-F(y)) \int_{-\infty}^y (x-y) f(x) {\rm d} x {\rm d} y \\
	&= 2 \int_{-\infty}^{\infty} f(x) \int_x^{\infty} (y-x) (1-F(y)) {\rm d} y {\rm d} x \\
	&= 2 \int_{-\infty}^{\infty} f(x) \int_{x}^{\infty} \frac{(y-x)^2}{2} f(y) {\rm d} y {\rm d} x \\
	&= \int_{-\infty}^{\infty} \int_{x}^{\infty} (x-y)^2 f(x) f(y) {\rm d} y {\rm d} x \\
	&= \frac{1}{2} \int_{-\infty}^{\infty} \int_{-\infty}^{\infty} (x-y)^2 f(x) f(y) {\rm d} y {\rm d} x. 
	\end{align*}
	Therefore, letting $X$ and $Y$ be independent samples from $f(z)$,
	\begin{align*}
	\Sigma_{11} &= \frac{1}{2} {\rm E} [(X-Y)^2] = m_2.
	\end{align*}
	A similar calculation yields
	\begin{align*}
	\Sigma_{12} = \Sigma_{21} = {\rm E} \left[ \frac{1}{3} X^3 - \frac{1}{2} X^2 Y + \frac{1}{6} Y^3 \right] = \frac{1}{2} m_3,
	\end{align*}
	\begin{align*}
	\Sigma_{22} = {\rm E} \left[ \frac{(X^2-Y^2)^2}{8} \right] = \frac{1}{4} (m_4-1).
	\end{align*}
	Hence, we obtain \eqref{asymp}.
\end{proof}

In particular, the $W$-estimator is Fisher efficient for the Gaussian model, but it is not efficient for other models.

\begin{corollary}
	For the Gaussian model, the asymptotic distribution of the $W$-estimator $(\hat{\mu},\hat{\sigma})$ is 
	\begin{align*}
	\sqrt{n} \begin{pmatrix} \hat{\mu}-\mu \\ \hat{\sigma}-\sigma \end{pmatrix} \to N \left( \begin{pmatrix} 0 \\ 0 \end{pmatrix}, \begin{pmatrix} \sigma^2 & 0 \\ 0 & \frac{1}{2} \sigma^2 \end{pmatrix} \right),
	\end{align*}
	which attains the Cramer--Rao bound.
\end{corollary}
\begin{proof}
	For the Gaussian model, we have $m_4=3$ and $m_3=0$.
\end{proof}

Figure~\ref{fig3} plots the ratio of the mean square error ${\rm E}[(\hat{\mu}-\mu)^2+(\hat{\sigma}-\sigma)^2]$ of the $W$-estimator to that of the MLE for the Gaussian model with respect to $n$.
The ratio converges to one as $n$ goes to infinity, which shows that the $W$-estimator has statistical efficiency.

\begin{figure}[htbp]
\centering
\begin{tikzpicture}
\tikzstyle{every node}=[]
\begin{axis}[width=7cm,
xmax=6,xmin=2,
ymax=1.008, ymin=0.999,
xlabel={$\log_{10} n$},
ylabel={MSE ratio},
      yticklabel style={
        /pgf/number format/.cd,
        fixed,
        zerofill,
        precision=3,
      },
ytick distance=0.002,
ylabel near ticks,
legend style={legend cell align=left,draw=none,fill=white,fill opacity=0.8,text opacity=1,},
	]
\addplot[very thick, color=black,
filter discard warning=false, unbounded coords=discard
] table {
    2.0000    1.0072
    3.0000    1.0014
    4.0000    1.0001
    5.0000    1.0000
    6.0000    1.0000
};
\addplot[thin, color=black, dotted,
filter discard warning=false, unbounded coords=discard
] table {
    2.0000    1.0000
    6.0000    1.0000
};
\end{axis}
\end{tikzpicture} 
	\caption{Ratio of mean square error of $W$-estimator to that of MLE for the Gaussian model.}
	\label{fig3}
\end{figure}
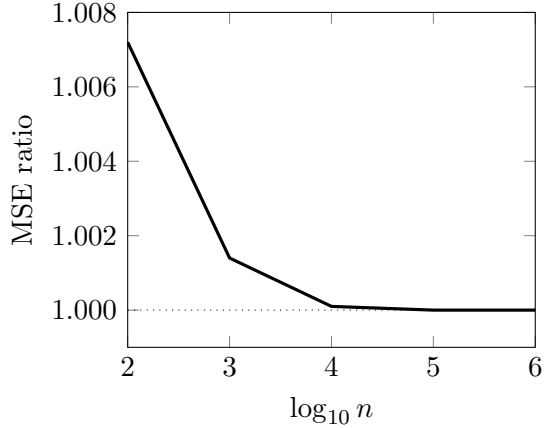

Figure~\ref{fig4} compares the mean square error of the $W$-estimator and MLE for the uniform model
\begin{align*}
f(z)= \begin{cases} \frac{1}{2 \sqrt{3}} & (-\sqrt{3} \leq z \leq \sqrt{3}) \\ 0 & (\mathrm{otherwise}) \end{cases}.
\end{align*}
In this case, the convergence rate of MLE is faster than $n^{-1/2}$, whereas the $W$-estimator is only $\sqrt{n}$-consistent.

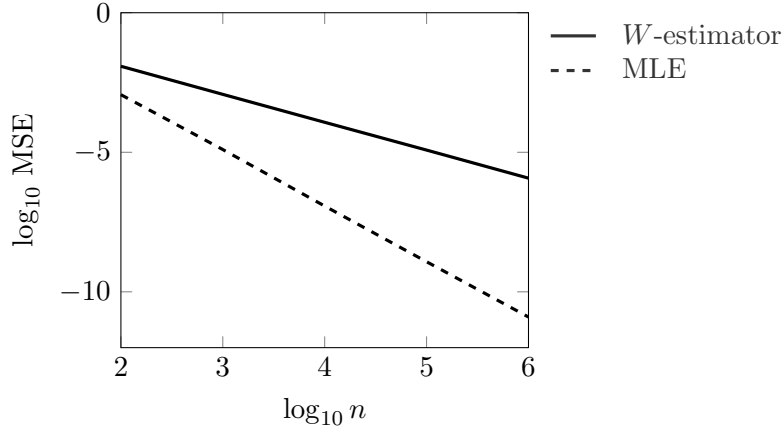
\begin{figure}[htbp]
\centering
\begin{tikzpicture}
\tikzstyle{every node}=[]
\begin{axis}[width=7cm,
xmax=6,xmin=2,
ymax=0, ymin = -12,
xlabel={$\log_{10} n$},ylabel={$\log_{10}$ MSE},
ylabel near ticks,
legend pos=outer north east,
legend entries={$W$-estimator,MLE},
legend style={legend cell align=left,draw=none,fill=white,fill opacity=0.8,text opacity=1,},
	]
\addplot[very thick, color=black,
filter discard warning=false, unbounded coords=discard
] table {
    2.0000   -1.9183
    3.0000   -2.9248
    4.0000   -3.9262
    5.0000   -4.9216
    6.0000   -5.9274
};
\addplot[very thick, dashed, color=black,
filter discard warning=false, unbounded coords=discard
] table {
    2.0000   -2.9360
    3.0000   -4.9022
    4.0000   -6.9269
    5.0000   -8.9223
    6.0000  -10.9029
};
\end{axis}
\end{tikzpicture} 
	\caption{Mean square error of $W$-estimator and MLE for the uniform model.}
	\label{fig4}
\end{figure}


\section{Riemannian structure of $W$-divergence}

Consider the manifold $M = \{p(x) \}$ of probability distributions which are absolutely continuous with respect to the Lebesgue measure and have finite second moments.  It is known that $M$ has a Riemannian structure due to the Wasserstein distance or the cost function. For two distributions $p(x)$ and $q(x)$, their optimal transportation cost, that is, the divergence between them, is given by (\ref{eq:am420200228}).

We calculate the optimal transportation cost between two nearby distributions $p(x)$ and $p(x)+ \delta p(x)$, where $\delta p(x)$ is infinitesimally small.  We have
\begin{align*}
\left(P + \delta P \right)^{-1} (u)
= P^{-1}(u)- 
\frac{\delta P \left\{ x(u) \right\}}{P' \left\{ x(u) \right\}},
\end{align*}
where
\begin{align*}
x(u) = P^{-1}(u).
\end{align*}
This equation is derived from
\begin{align*}
\frac {\rm d}{{\rm d}u} F^{-1}(u) =
\frac 1{f \left\{ x(u) \right\}},
\end{align*}
which comes from the differentiation of the identity
\begin{align*}
F^{-1} \left\{ F(x) \right\} = x.
\end{align*}
We thus have
\begin{align}
\label{eq:am6920200225}
C \left(p, p+\delta p \right) =
\int^{\infty}_{-\infty}
\frac 1{p(x)} \left(
\int^{x}_{-\infty} \delta p(y)dy
\right)^2 dx
\end{align}
which is a quadratic form of $\delta p(x)$.  This gives a Riemannian metric to $M$.

The location-scale model $S$ is a finite-dimensional submanifold embedded in $M$.  
For the location-scale model (\ref{eq:am3120200225}), we have
\begin{align*}
\delta p(y) = \frac{\partial}{\partial \mu}
p(y, {\bm{\theta}}) d \mu +
\frac{\partial}{\partial \sigma}
p(y, {\bm{\theta}}) d \sigma.
\end{align*}
The Riemannian metric tensor $G^W=\left(g^W_{ij}\right)$ is derived from
\begin{align*}
C (p, p+\delta p) = \sum g_{ij}^W
({\bm{\theta}}) d \theta_i d \theta_j.
\end{align*}
See also \citet*{LZ2019}.

\begin{theorem}\upshape
	The location-scale model is a Euclidean space, irrespective of $f$,
	\begin{align*}
	g_{ij}^W = \delta_{ij}.
	\end{align*}
\end{theorem}

\begin{proof}\upshape
	We need to calculate (\ref{eq:am6920200225}).  
	We have
	\begin{align*}
	\delta p(x, {\bm{\theta}}) = 
	-\frac 1{\sigma^2} f' \left( 
	\frac{x-\mu}{\sigma} \right) d \mu
	-\frac 1{\sigma^3} \left\{
	\sigma f \left( \frac{x-\mu}{\sigma} \right)
	+ (x-\mu) f' \left( \frac{x-\mu}{\sigma} \right)
	\right\} d \sigma.
	\end{align*}
	Integration gives
	\begin{align*}
	\int^{x}_{-\infty} \delta p
	(y, {\bm{\theta}}) dy =
	-p(x, {\bm{\theta}}) d \mu -
	\frac{1}{\sigma} (x-\mu) p(x, {\bm{\theta}}) d \sigma.
	\end{align*}
	Hence, we have
	\begin{align*}
	C({\bm{\theta}}, {\bm{\theta}}+ d{\bm{\theta}})
	= d \mu^2 + d \sigma^2.
	\end{align*}

\end{proof}
It is surprising that $G=\left(g_{ij}\right)$ is the identity matrix for the location-scale model, so $S$ is a Euclidean space.  See also \citet{LZ2019}.  It is flat by itself, but $S$ is a curved submanifold in $M$ \citep{Takatsu2011}, like a cylinder embedded in ${\bm{R}}^3$.

When $n$ is large, the cost decreases on the order of $1/n$.  The $W$-estimator is the projection of $\hat{p}(x)$ to $S$ in the tangent space of $M$.  Let $\hat{{\bm{\theta}}}'$ be another consistent estimator. Accordingly, we have the Pythagorean relation
\begin{align*}
C \left(\hat{p}, p_{\hat{{\bm{\theta}}}'} \right)
= C \left(\hat{p}, p_{\hat{{\bm{\theta}}}} \right)
+ C \left(p_{\hat{{\bm{\theta}}}}, p_{\hat{{\bm{\theta}}}'} \right),
\end{align*}
and the difference of the cost between the two estimators is 
\begin{align*}
C \left( p_{\hat{{\bm{\theta}}}}, p_{\hat{{\bm{\theta}}}'} \right)
= \frac{1}{n} \left|\hat{\bm{\theta}}- \hat{\bm{\theta}}' \right|^2.
\end{align*}

\citet*{LZ2019} studied the properties of a $W$-estimator given by the $W$ score function. They gave the $W$-efficiency and $W$-Cramer-Rao inequality.  However, their $W$-estimator does not minimize the transportation cost.  

\section{Maximum likelihood estimator and $W$-divergence}
It is an interesting problem to study the estimator that minimizes the transportation cost from the true distribution to the estimated one. 
Let $\hat{{\bm{\theta}}}$ be a consistent estimator and let ${\bm{e}}=\hat{{\bm{\theta}}}-{\bm{\theta}}_0$ be the estimation error vector, where ${\bm{\theta}}_0$ is the true parameter. 
We want to study the minimizer of $C(p_{{\bm{\theta}}_0},p_{\hat{{\bm{\theta}}}})$. 
Since the W-metric $g$ is the identity matrix for the location scale model, for the covariance $V=E[(\hat{{\bm{\theta}}}-{\bm{\theta}}_0)(\hat{{\bm{\theta}}}-{\bm{\theta}}_0)^{\top}]$ of the estimation error, we have
\[
C = {\rm tr} V.
\]
Therefore, the covariance is minimized when the expectations of the sum of the squares of the location error and scale error are at a minimum in the location scale case. Furthermore, we have a more general result.

\begin{theorem}
	The transportation cost is asymptotically minimized  by the maximum likelihood estimator for a general statistical model.
\end{theorem}
\begin{proof}
	The error covariance $V$ satisfies the Cramer--Rao inequality
	\[
	V \succeq \frac{1}{n} G_F^{-1}
	\]
	in the sense of the matrix positive-definiteness, where $G_F$ is the Fisher information matrix.
	The minimum is attained asymptotically by the MLE. 
	On the other hand, when $A \succeq B$ for two positive-definite matrices $A$ and $B$,
	\[
	{\rm tr} (G^W A) \geq {\rm tr} (G^W B).
	\]
	Since the transportation cost is asymptotically written as
	\[
	C = {\rm tr} (G^W V) \geq \frac{1}{n} {\rm tr} (G^W G_F^{-1}),
	\]
	it is minimized for the maximum likelihood estimator that asymptotically attains $V=G_F^{-1}/n$.
\end{proof}

It would be interesting to analyze the transportation cost of the $W$-estimator in general.

\section{Discussion}

There are three estimators, the MLE, $W$-score estimator and $W$-estimator. They have their own optimal properties and related behaviors. The MLE minimizes the KL divergence from the empirical distribution to the estimated distribution in the model. It minimizes the KL divergence and the $W$-divergence (transportation cost) from the true distribution to the estimated model at the same time. The $W$-estimator minimizes the transportation cost from the empirical distribution to the estimated distribution. However, it does not necessarily minimize the cost from the true distribution to the estimated one. The $W$-score estimator minimizes the integrated W-score function which is not the transportation cost. Further studies should be conducted on the merits and demerits of these estimators and their applicability to various problems.

\end{document}